\documentclass[12pt,a4paper]{article}%
\usepackage[utf8]{inputenc}
\usepackage{hyperref}
\usepackage{amsmath}
\usepackage{amsfonts}
\usepackage{amssymb}
\usepackage{graphicx}%
\providecommand{\U}[1]{\protect\rule{.1in}{.1in}}
\newtheorem{theorem}{Theorem}

\newtheorem{conjecture}[theorem]{Conjecture}
\newtheorem{corollary}[theorem]{Corollary}

\newtheorem{lemma}[theorem]{Lemma}

\newtheorem{problem}[theorem]{Problem}
\newtheorem{proposition}[theorem]{Proposition}

\newenvironment{proof}[1][Proof]{\noindent\textbf{#1.} }{\ \hfill \rule{0.5em}{0.5em}\bigskip}
\graphicspath{{E:/Dropbox/Riste-Sedlar/MetricDIm/Jelena/Slike/}}

\begin{document}

\title{Bounds on metric dimensions of graphs with edge disjoint cycles}
\author{Jelena Sedlar$^{1}$,\\Riste \v Skrekovski$^{2,3}$ \\[0.3cm] {\small $^{1}$ \textit{University of Split, Faculty of civil
engineering, architecture and geodesy, Croatia}}\\[0.1cm] {\small $^{2}$ \textit{University of Ljubljana, FMF, 1000 Ljubljana,
Slovenia }}\\[0.1cm] {\small $^{3}$ \textit{Faculty of Information Studies, 8000 Novo
Mesto, Slovenia }}\\[0.1cm] }
\maketitle

\begin{abstract}
In a graph $G$, cardinality of the smallest ordered set of vertices that
distinguishes every element of $V(G)$ is the (vertex) metric dimension of $G$.
Similarly, the cardinality of such a set is the edge metric dimension of $G$,
if it distinguishes $E(G)$. In this paper these invariants are considered
first for unicyclic graphs, and it is shown that the vertex and edge metric
dimensions obtain values from two particular consecutive integers, which can
be determined from the structure of the graph. In particular, as a
consequence, we obtain that these two invariants can differ for at most one
for a same unicyclic graph. Next we extend the results to graphs with edge
disjoint cycles showing that the two invariants can differ at most by $c$,
where $c$ is the number of cycles in such a graph. We conclude the paper with
a conjecture that generalizes the previously mentioned consequences to graphs
with prescribed cyclomatic number $c$ by claiming that the difference of the
invariant is still bounded by $c$.

\end{abstract}



\section{Introduction}

Here we consider only simple and connected graphs. In a graph $G$, we denote
by $d_{G}(u,v)$ (or simply $d(u,v)$ if no confusion arises) the distance
between two vertices $u,v\in V(G)$. Now, if $d_{G}(u,s)\ne d_{G}(v,s)$, for
some vertices $s,u,v$ of $G$, then we say that $s$ \emph{distinguishes} (or
\emph{resolves}) $u$ and $v$. If any two vertices $u$ and $v$ are
distinguished by at least one vertex of a subset $S\subseteq V(G)$, then we
say that $S$ is a \emph{metric generator} for $G$. The cardinality of the
smallest metric generator is called the \emph{metric dimension} of $G$, and it
is denoted by $\dim(G)$. This notion for graphs was independently introduced
by \cite{Harary1976} and \cite{Slater1975}, under the names resolving sets and
locating sets, respectively. Even before this notion was introduced for the
realm of metric spaces \cite{Blumenthal1953}. In the paper, as we deal wih
several types of metric dimension, in order to emphasize that we deal in
certain situation with the usual metric dimension, we use the word "vertex" as
a prefix, and say a vertex metric generator and a vertex metric dimension.

The concept of metric dimension was recently extended from resolving vertices
to resolving edges of a graph by Kelenc, Tratnik and Yero~\cite{Kel}.
Similarly as above, a vertex $s\in V(G)$ \emph{distinguishes} two edges
$e,f\in E(G)$ if $d_{G}(s,e)\neq d_{G}(s,f)$, where $d_{G}(e,s)=d_{G}%
(uv,s)=\min\{d(u,s),d(v,s)\}$. A set of vertices $S\subseteq V(G)$ is an
\emph{edge metric generator} for $G$, if any two edges of $G$ are
distinguished by a vertex of $S$. The cardinality of the smallest edge metric
generator is called the \emph{edge metric dimension} of $G$, and it is denoted
by $\mathrm{edim}(G)$. In~\cite{Kel} it was shown that determining the edge
metric dimension of a graph is NP-hard. Also for trees, grid graphs, wheels
and some other graph classes are given bounds and closed formulas. In
particular, families of graphs for which $\dim(G)<\mathrm{edim}(G)$, or
$\dim(G)=\mathrm{edim}(G)$, or $\dim(G)>\mathrm{edim}(G)$ were presented.

Edge metric dimension immediately attracted big attention. In
\cite{Peterin2020}, Peterin and Yero were considering the edge metric
dimension of corona, join and lexicographic products of graphs. In~\cite{Zhu}
was considered the maximum possible value of edge metric dimension amongst
graphs of prescribed order. And, Zubrilina~\cite{Zubrilina} showed that it is
not posisble to bound the metric dimension of a graph $G$ by some function of
the edge metric dimension of $G$.

For a wider and systematic introduction of the topic of metric dimension, we
recommend the PhD thesis of Kelenc~\cite{KelPhD}. All three
works~\cite{KelPhD,Kel,Kelm} propose various open problems and research
directions for possible further work.

Throughout the paper we will use the following notation. The only cycle in the
unicyclic graph $G$ that is under our consideration is denoted by
$C=v_{0}v_{1}\cdots v_{g-1}$, where $g$ is the length of $C$ (i.e.
$g=\left\vert V(C)\right\vert $). The connected component of $G-E(C)$
containing vertex $v_{i}$ is denoted by $T_{v_{i}}$.

A \emph{thread} in a graph $G$ is a path $u_{1}u_{2}\cdots u_{k}$ in which all
vertices are of degree $2$ except for $u_{k}$ which is of degree $1$ and
$u_{1}$ is a neighbour of a vertex $v\in V(G)$ with $\deg(v)\geq3$. Note that
vertices of the only cycle in a unicyclic can have acyclic structures attached
to them by an edge. When such acyclic structure does not contain a vertex of
degree $\geq3$ then it is a thread, otherwise we call such structure a\emph{
branch} (for illustration see Figure \ref{Figure15}). When a branch is
attached to a vertex of the cycle, then there is certainly a pair of vertices
and a pair of edges in the branch which are on the same distance from the
cycle $C$ which, therefore, cannot be distinguished by a vertex from outside
the branch. Note that the same branching phenomenon occurs when there are two
(or more) threads attached to the same vertex of the cycle and no branches.
This is of interest to us, so we introduce the following definition. For a
vertex $v$ from a unicyclic graph $G$ we say that it is a \emph{branching}
vertex if $v\not \in V(C)$ and $\deg(v)\geq3$ or if $v\in V(C)$ and
$\deg(v)\geq4$. We say that a vertex $v_{i}\in V(C)$ is \emph{branch-active}
if $T_{v_{i}}$ contains a branching vertex. Note that a vertex $v_{i}$ from
the cycle $C$ is branch-active if there is a branch hanging at $v_{i}$ but
also if there is more than one thread attached to it. Denote by $b(C)$ the
number of all branch-active vertices on $C$. As the cycle $C$ is the only
cycle in a unicyclic graph, we can use notation $b(G)$ instead of $b(C).$

\begin{figure}[h]
\begin{center}
\includegraphics[scale=1.0]{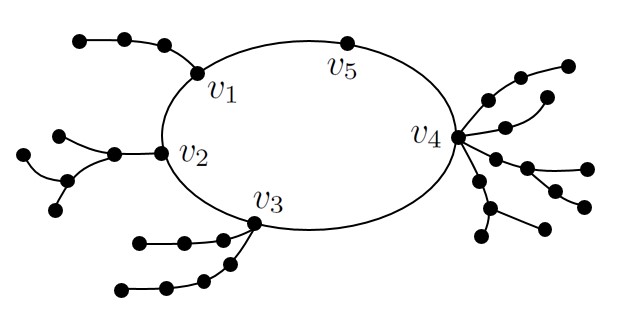}
\end{center}
\par
\caption{An illustration of threads and branches: $v_{1}$ has a thread
attached and $v_{2}$ has a branch attached. Vertices of the cycle can have
more than one acyclic structure attached to them ($v_{3}$ has two threads
attached and $v_{4}$ has two threads and two branches attached) or none
(vertex $v_{5}$). Branch-active vertices are $v_{2}$, $v_{3}$ and $v_{4}$
(even though $v_{3}$ does not have a branch attached). Finally, $T_{v_{1}}$
and $T_{v_{2}}$ consist of a thread and a branch attached to a respective
vertices, but $T_{v_{3}}$ consists of both threads attached to $v_{3}$ and the
same holds for $T_{v_{4}}$ - it consists of all threads and branches attached
to $v_{4}.$ Finally, $T_{v_{5}}$ is trivial, it consists only from the vertex
$v_{5}.$}%
\label{Figure15}%
\end{figure}

We say that a subgraph $H$ of a graph $G$ is an \emph{isometric} subgraph, if
for any two vertices $u,v\in V(H)$ it holds that $d_{H}(u,v)=d_{G}(u,v)$. The
following notation for paths is used. Suppose that $P$ is a path and $u,v\in
V(P),$ then by $P[u,v]$ we denote the subpath of $P$ connecting vertices $u$
and $v,$ while by $P(u,v)$ we denote $P[u,v]-\{u,v\}$. Notions $P[u,v)$ and
$P(u,v]$ are also used and they denote the subpaths where only one of the
end-vertices of $P[u,v]$ is excluded. For any edge $e\in E(G)$, let $G/e$ be
the graph obtained from $G$ by contracting $e$.

\section{Branch-resolving sets}

In the paper, we will establish lower and upper bounds on metric dimensions
for unicyclic graphs. As for considering the lower bound, we define that a set
$S\subseteq V(G)$ of a graph $G$ is \emph{branch-resolving} if for every $v\in
V(G)$ of degree at least $3$, the set $S$ contains a vertex from all threads
starting at vertex $v$ except possibly from one such thread. In this short
section, we provide several properties of branch-resolving sets. Notice that
one can always choose a branch-resolving set comprised of leaves.

Let us denote by $\ell(v)$ the number of all threads attached to a vertex $v$
of $G$, and let
\[
L(G)=\sum_{v\in V(G),\ell(v) >1}(\ell(v)-1).
\]
Note that for every branch-resolving set $S$ we have $\left\vert S\right\vert
\geq L(G)$ with equality holding for branch-resolving sets of minimum
cardinality. Regarding the trees, the following nice result is well known,
see~\cite{Kel, Khu}.

\begin{proposition}
For every tree $T$ that is not a path, it holds
\[
\dim(T) =\mathrm{edim}(T) =L(T).
\]

\end{proposition}

Let $G$ be a unicyclic graph, let $S\subseteq V(G)$ be a set of vertices in
$G$, and let $v_{i}$ be a vertex from $C$. We say that a vertex $v_{i}$ is
\emph{$S$-active}, if $T_{v_{i}}$ contains a vertex from $S$. We will mainly
be interested in $S$-active vertices with respect to a given branch-resolving
set $S$. Let $A(S)\subseteq V(C)$ be the set of vertices in $V(C)$ which are
$S$-active, also let $a(S)$ be the cardinality of the set $A(S)$.

\begin{lemma}
\label{l.br1} Let $S$ be a metric generator or an edge metric generator of a
unicyclic graph $G$. Then $S$ is a branch-resolving set with $a(S)\ge2$.
\end{lemma}

\begin{proof}
If $S$ is not a branch-resolving set, there is a vertex $v\in V(G)$ of degree
at least three such that there are two threads $x_{1}x_{2}\cdots x_{k}$ and
$y_{1}y_{2}\cdots y_{l}$ attached to $v$ (vertices $x_{1}$ and $y_{1}$ are
adjacent to $v$) such that $S$ does not contain a vertex from these two
threads. Let $G_{v}$ be the connected component of $G-\{vx_{1},vy_{1}\}$ which
contains $v$. Note that $S\subseteq V(G_{v})$ and $d(x_{1},z)=d(y_{1},z)$ for
every $z\in V(G_{v})$. Therefore, $S$ does not distinguish $x_{1}$ and $y_{1}%
$, and so $S$ is not a metric generator. Also, note that $d(x_{1}%
v,z)=d(y_{1}v,z)$ for every $z\in V(G_{v}),$ which implies that $S$ does not
distinguish edges $x_{1}v$ and $y_{1}v,$ so neither is $S$ an edge metric
generator. In both cases we obtain a contradicton.

If $a(S)=0$ then $S=\phi$ and the claim is obvious. If $a(S)=1$, then let
$v_{i}$ be the vertex from cycle $C$ in $G$ that is $S$-active. Note that for
the connected component $T_{v_{i}}$ it holds that $S\subseteq V(T_{v_{i}})$.
Let $v_{i-1}$ and $v_{i+1}$ be two neighbors of $v_{i}$ on $C$. Then,
$v_{i-1}$ and $v_{i+1}$ are not distinguished by $S$, so $S$ is not a metric
generator. Also, edges $v_{i-1}v_{i}$ and $v_{i}v_{i+1}$ are not distinguished
by $S$, so $S$ is not an edge metric generator either.
\end{proof}

Above result tells that every generating set is branch-resolving, the opposite
does not hold but still by a branch-resolving set we can distinguish "local"
pairs of vertices and edges as it is shown in the next lemma.

\begin{lemma}
\label{Prop_SameComponent} Let $G$ be a unicyclic graph and $S\subseteq V(G)$
a branch-resolving set with $a(S)\geq2$. Then, any two vertices (also any two
edges) from a same connected component of $G-E(C)$ are distinguished by $S$.
\end{lemma}

\begin{proof}
Since $a(S)\geq2,$ there are at least two vertices on the cycle $C$ which are
$S$-active, say $v_{i}$ and $v_{j}$. Let $x$ and $x^{\prime}$ be two vertices
from a same component of $G-E(C)$, say $x$ and $x^{\prime}$ both belong to
$T_{v_{k}}$. Without loss of generality, we may assume that $k\not =i$. Then
$x$ and $x^{\prime}$ are distinguished by a vertex $s\in S\cap V(T_{v_{i}})$
in all cases except when $d(x,v_{k})=d(x^{\prime},v_{k})$. Therefore, suppose
that $d(x,v_{k})=d(x^{\prime},v_{k})$ and let $P$ be the only path in $G$
connecting vertices $x$ and $x^{\prime}$. Denote by $v$ the middle vertex of
the path $P$ which must exist because of our assumption that $d(x,v_{k}%
)=d(x^{\prime},v_{k})$, and so $P$ is of even length. Let $G_{v}$ be the
connected component of $G-E(P)$ containing vertex $v$. Note that $x$ and
$x^{\prime}$ are distinguished by all vertices outside $G_{v}$. So, outside
$G_{v}$ there cannot be any vertex from $S$.

Let $G_{x}$ and $G_{x^{\prime}}$ be the connected components of $G-v$
containing vertices $x$ and $x^{\prime}$ respectively. If both $G_{x}$ and
$G_{x^{\prime}}$ are threads, then vertex $v$ has two threads attached to it
which do not contain a vertex from $S$ which is a contradiction. Therefore,
$G_{x}$ or $G_{x^{\prime}}$ is a branch i.e. it contains at least one vertex
of degree $\geq3$, say $G_{x}$ is a branch. Therefore, $G_{x}$ must contain at
least two threads attached to the same vertex which do not contain a vertex
from $S$ which is a contradiction.

Suppose now that $e=xy$ and $e^{\prime}=x^{\prime}y^{\prime}$ are two edges
from the same $T_{v_{k}}$. We may assume that $x$ is closer to $v_{k}$ than
$y$, and similarly $x^{\prime}$ is closer to $v_{k}$ than $y^{\prime}$. We may
also assume $d(x,v_{k})=d(x^{\prime},v_{k})$, otherwise $e$ and $e^{\prime}$
will be distinguished by any vertex $s\in S\backslash V(T_{v_{k}})$. Now, let
$P$ be the shortest path connecting vertices $x$ and $x^{\prime}$. Note that
the fact that $d(x,v_{k})=d(x^{\prime},v_{k})$ implies $P$ is of even length
which means there is a vertex $v$ sitting in the middle of $P$ which is of
degree at least $3$. Note that $e$ and $e^{\prime}$ are distinguished by $S$
unless $S\subseteq V(G_{v})$ where $G_{v}$ is the connected component of
$G-E(P)$ containing vertex $v.$ Suppose therefore that $S\subseteq V(G_{v})$.
Observe that connected components $G_{y}$ and $G_{y^{\prime}}$ of $G-v$ that
contains $y$ and $y^{\prime}$, respectively, are different. If both $G_{y}$
and $G_{y^{\prime}}$ are threads, then there are two threads attached to $v$
which do not contain a vertex from $S$ which is a contradiction. Otherwise,
$G_{y}$ or $G_{y^{\prime}}$ is a branch, then it must contain a vertex from
$S$ and consequently $S\not \subseteq V(G_{v})$, a contradiction.
\end{proof}


\section{Geodesic triples}

Branch-resolving sets enable us to distinguish vertices and edges from same
$T_{v_{i}}$'s. Now, we introduce "small" sets which enable us to distinguish
vertices and edges from distinct $T_{v_{i}}$'s. Let $v_{i}$, $v_{j}$, and
$v_{k}$ be three vertices belonging to the only cycle $C$ in a unicyclic graph
$G$. We say that $v_{i}$, $v_{j}$, and $v_{k}$ form a geodesic triple of
vertices on $C$, if
\[
d(v_{i},v_{j})+d(v_{j},v_{k})+d(v_{i},v_{k})=|V(C)|.
\]
Observe that for any two vertices of $C$, we can easily choose a thrid one
such that they form a geodesic triple. Let us now proceed to prove that
geodesic triple of vertices distinguishes all pairs of vertices and all pairs
of edges which are not in the same connected component of $G-E(C)$. Notice
that the edges of $C$ are also considered here.

\begin{lemma}
\label{l.gt.v} Let $G$ be a unicyclic graph and let $S$ be a geodesic triple
of vertices from $C$. Then, $S$ distinguishes any two vertices that belong to
two distinct components of $G-E(C)$.
\end{lemma}

\begin{proof}
Suppose to the contrary, i.e. $S$ does not distinguish vertices $x$ and
$x^{\prime}$ that belong to distinct components of $G-E(C)$. We may assume
that $x$ belongs to $T_{v_{i}}$ and $x^{\prime}$ belongs to $T_{v_{j}}$ with
$i\ne j$. Let $v_{j}^{\ast}$ be the antipodal of $v_{j}$ such that in case
when $g$ is odd, it has two antipodals, we choose the one closer to $v_{i}$.
And similarly, let $v_{i}^{\ast}$ be the antipodal of $v_{i}$ that is closer
to $v_{j}$. See Figure~\ref{f.l5}, where it is assumed that $C$ is odd,
$v_{i}$ (resp. $v_{j}$) is connected to its two antipodal vertices, the one
that is closer to $v_{j}$ (resp. $v_{i}$) is denoted by $v_{i}^{*}$ (resp.
$v_{j}^{*}$).

\begin{figure}[h]
\begin{center}
\includegraphics[scale=1.0]{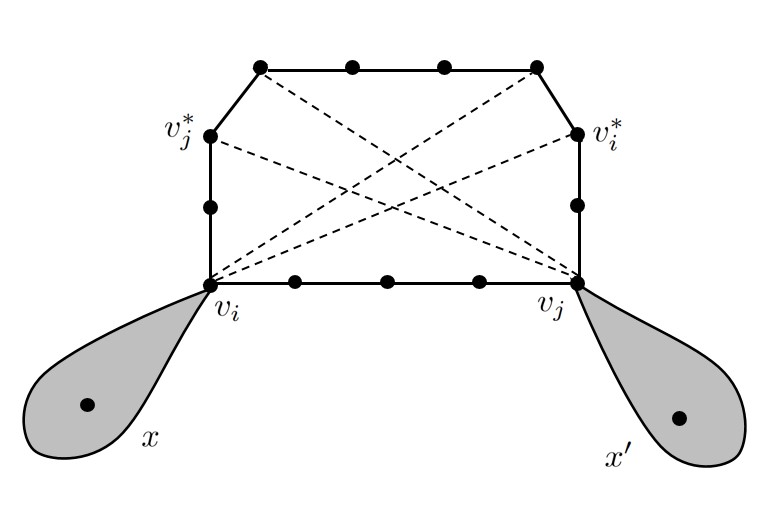}
\end{center}
\par
\caption{An illustration of the situation in Lemma~\ref{l.gt.v} with $C$ being
of odd length: vertex $v_{i}$ has two antipodal vertices where we choose for
$v_{i}^{\ast}$ the one of the two closer to $v_{j}$ and the same holds for
$v_{j}$.}%
\label{f.l5}%
\end{figure}

Let $P_{1}$ and $P_{2}$ be the two paths in $G$ connecting $x$ and $x^{\prime
}$. If $P_{1}$ and $P_{2}$ are of different length, we assume that $P_{1}$ is
shorter than $P_{2}$. Thus, $v_{i}^{*}$ and $v_{j}^{*}$ belong to $P_{2}$. In
what follows, we show several claims regarding $P_{1}, P_{2}$, and $S$ in
order to complete the proof.

\medskip\noindent\textbf{Claim 1.} \emph{Each of $P_{1}[v_{i},v_{j}]$ and
$P_{2}(v_{i}^{\ast},v_{j}^{\ast})$ contains at most one vertex from $S$.}

\medskip\noindent Regarding $P_{1}[v_{i},v_{j}]$, this is due to the fact
$P_{1}$ is a shortest path between $x$ and $x^{\prime}$ which implies it is an
isometric path. Now, the claim follows from the fact that an isometric path
cannot contain two distinct vertices on equal distance to one of its end-vertices.

Suppose now that $P_{2}(v_{i}^{\ast},v_{j}^{\ast})$ contains at least two
vertices from $S$, say $s_{1}$ and $s_{2}$. The fact that $s_{\ell}$ ($\ell
\in\{1,2\}$) does not distinguish $x$ and $x^{\prime}$ implies
\[
d(s_{\ell},v_{i})+d(v_{i},x)=d(s_{\ell},v_{j})+d(v_{j},x^{\prime}),
\]
or equivalently
\[
d(s_{\ell},v_{i})-d(s_{\ell},v_{j})=d(v_{j},x^{\prime})-d(v_{i},x).
\]
As the right side of the equality is the same for both $\ell=1$ and $\ell=2$,
we promptly derive that
\[
d(s_{1},v_{i})-d(s_{1},v_{j})=d(s_{2},v_{i})-d(s_{2},v_{j}).
\]
But this is not possible due to the fact that $P_{2}[v_{j},v_{j}^{\ast}]$ and
$P_{2}[v_{i},v_{i}^{\ast}]$ are isometric paths, and all these four distances
are realized by subpaths of $P_{2}$. This establishes the claim.

\medskip\noindent\textbf{Claim 2.} \emph{If $P_{2}[v_{i},v_{j}^{\ast}]$
contains a vertex from $S$ then $P_{2}(v_{j}^{\ast},v_{j}]$ contains no vertex
from $S$.} Suppose the claim is false and $s_{1}$ is in $P_{2}[v_{i}%
,v_{j}^{\ast}]$ and $s_{2}$ is in $P(v_{j}^{\ast},v_{j}]$. Since $s_{1}$ does
not distinguish $x$ and $x^{\prime}$, it follows that $d(v_{i},x)=d(v_{i}%
,v_{j})+d(v_{j},x^{\prime})$. But now%

\[%
\begin{array}
[c]{rcl}%
d(s_{2},x) & = & d(s_{2},v_{i})+d(v_{i},x)\\
& = & d(s_{2},v_{i})+d(v_{i},v_{j})+d(v_{j},x^{\prime})\\
& > & d(s_{2},v_{j})+d(v_{j},x^{\prime})\\
& = & d(s_{2},x^{\prime}).
\end{array}
\]
This implies that $s_{2}$ (and so $S$) disitinguishes $x$ and $x^{\prime}$,
which is a contradiction.

\medskip\noindent\textbf{Claim 3.} \emph{If $P_{2}[v_{j},v_{i}^{*}]$ contains
a vertex from $S$ then $P_{2}(v_{i}^{*},v_{i}]$ contains no vertex from $S$.}
The proof is similar as in Claim 3.

\medskip Now, we apply the above claims in order to conclude the proof. Notice
that if $P_{2}[v_{i},v_{j}^{\ast}]$ contains a vertex from $S$ then due to
Claims 2 and 3, all vertices of $S$ are contained in $P[v_{j},v_{j}^{\ast}]$.
Since $S$ contains a geodesic triple this is possible only if $g$ is even and
$v_{j}$ and $v_{j}^{\ast}$ belong to $S$. The fact that $v_{j}^{\ast}$ is
element of $S$ implies $d(v_{i},x)=d(v_{i},v_{j})+d(v_{j},x^{\prime})$ but
then $v_{j}\in S$ distinguishes $x$ and $x^{\prime}$, which is a
contradiction. We argue similarly if $P_{2}[v_{i}^{\ast},v_{j}]$ contains a
vertex from $S$.

On the other hand if there is no vertex of $S$ in $P_{2}[v_{i},v_{j}^{*}]$ and
$P_{2}[v_{i}^{*},v_{j}]$, then by Claim 1, $S$ can contain at most one vertex
in $P_{2}(v_{i}^{*},v_{j}^{*})$ and in $P_{1}(v_{i},v_{j})$ which is a
contradiction with the assumption that $S$ is of size at least three.
\end{proof}

Now we show the edge version of the previous lemma.

\begin{lemma}
\label{l.gt.e} Let $G$ be a unicyclic graph and let $S$ be a geodesic triple
of vertices from $C$. Then, $S$ distinguishes any two edges that does not
belong to a same component of $G-E(C)$.
\end{lemma}

\begin{proof}
Suppose to the contrary that $S$ does not distinguish two edges that belong to
two distinct components of $G-E(C)$, say $e=xy$ and $e^{\prime}=x^{\prime
}y^{\prime}$. If $C$ is of length $3$, then $S$ contains all vertices of $C$,
and consequently promptly follows that $S$ disitnguishes $e$ and $e^{\prime}$.
So, we assume that $C$ is of length $\ge4$.

We will consider three cases regarding whether $e$, $e^{\prime}$ belong to $C$
in order to conlcude the proof.

Suppose first that neither $e$ nor $e^{\prime}$ belongs to $C$. We may assume
that $x$ is closer to $C$ than $y$, and similarly $x^{\prime}$ is closer to
$C$ than $y^{\prime}$. Then $d(e,s)=d(x,s)$ and $d(e^{\prime},s)=d(x^{\prime
},s)$ for every $s\in S$. Now, Lemma~\ref{l.gt.v} assures that a vertex from
$S$ distinguishes $x$ and $x^{\prime}$, and so it distinguishes $e$ and
$e^{\prime}$ as well.

Suppose now that both $e$ and $e^{\prime}$ belong to $C$. Let $P_{x}$ and
$P_{y}$ be the paths from $C-e-e^{\prime}$. We may assume that end-vertices of
$P_{x}$ are $x$ and $x^{\prime}$ and the end-vertices of $P_{y}$ are $y$ and
$y^{\prime}$. Notice that a vertex $s$ from $P_{x}$ does not distinguish $e$
and $e^{\prime}$ only if $P_{x}$ is of even length and $s$ sits in the middle
of it. Similarly holds for $P_{y}$. But as we have at least three vertices in
$S$, we conclude that the third one must distinguish these two edges.

Suppose now that $e$ belongs to $C$ and $e^{\prime}$ does not. We may assume
that $x^{\prime}$ is closer to $C$ than $y^{\prime}$. In case $x^{\prime}=x$
note that there is a vertex $s\in S$ such that $d(s,e)=d(s,y)<d(s,x),$ since
otherwise $S$ would not contain geodesic triple. But then
$d(s,x)=d(s,x^{\prime})=d(s,e^{\prime})$ further implies $S$ distinguishes $e$
and $e^{\prime}$. Similarly argue when $x^{\prime}=y$. So, we can assume
$x^{\prime}$ is distinct from $x$ and $y$.

Next, consider the graph $G/e$ obtained by contracting the edge $e$ into a
vertex $v_{xy}$. As $C$ is of length $\ge4$, the new graph is also a unicyclic
graph. Notice that for any vertex $s$ from $C$, we have $d_{G}(s,e)=d_{G/e}%
(s,v_{xy})$ as the corresponding shortest paths belong to the cycles $C$ and
$C/e$, and they coincide.

Now, by Proposition~\ref{l.gt.e}, we have a vertex $s$ that distinguishes
$x^{\prime}$ and $v_{xy}$ in $G/e$. Consider two possibilities in order to
complete the proof. First, if $d_{G}(x^{\prime},s)=d_{G/e}(x^{\prime},s)$,
then the same $s$ distinguishes $x^{\prime}$ and $e$ in $G$, and hence
$e^{\prime}$ and $e$. The second possibility is when $d_{G}(x^{\prime
},s)=d_{G/e}(x^{\prime},s)+1$, i.e. the shortest path between $s$ and
$x^{\prime}$ goes through edge $e$, and it gets decreases by 1 after
contracting $e$. But in that case obviously it must hold $d(s,e)>d(s,x^{\prime
}) =d(s,e^{\prime})$ as $x$ and $y$ belong on the shortest path from $s$ to
$e^{\prime}$. This concludes the proof.
\end{proof}


\section{Branching-resolving sets vs. geodesic triples}

The last two lemmas will now help us to prove that even if a geodesic triple
of vertices from the cycle $C$ is not introduced into a branch-resolving set
$S$, but there are three $S$-active vertices on that cycle, then the set $S$
will distinguish all pairs of vertices and all pairs of edges. In other words,
we will prove that it does not matter if the vertex included in set $S$ is
vertex $v_{i}$ from the cycle or any vertex from inside the tree $T_{v_{i}},$
as long as we have a geodesic formation all pairs of vertices and all pairs of
edges are distinguished. Let us state and prove this formally.

\begin{lemma}
\label{Prop_gt} Let $G$ be a unicyclic graph and let $S$ be a branch-resolving
set of $G$ with $a(S)\geq3$ and there are three $S$-active vertices on $C$
forming a geodesic triple. Then, $S$ is both a metric generator and an edge
metric generator of $G$.
\end{lemma}

\begin{proof}
Let us first prove that $S$ is a metric generator. Suppose to the contrary,
i.e. there are two vertices $x,x^{\prime}\in V(G)$ which are not distinguished
by $S$. Without loss of generality we may assume that $x\in V(T_{v_{0}})$ and
$x^{\prime}\in V(T_{v_{i}})$ where $i\leq\left\lfloor g/2\right\rfloor $.
Since $S$ is a branch-resolving set with $a(S)\geq3$, Lemma
\ref{Prop_SameComponent} implies $i\not =0$. Now, Lemma \ref{l.gt.v} implies
that $x$ and $x^{\prime}$ are distinguished by a geodesic triple of $S$-active
vertices on $C$, and we want to prove that $x$ and $x^{\prime}$ are
distinguished by $S$ as well. Suppose vertices $x$ and $x^{\prime}$ are
distinguished by $S$-active vertex $v_{j}\in C$. If $j\not \in \{0,i\}$ then
for a vertex $s\in S\cap V(T_{v_{j}})$ the fact $d(x,v_{j})\not =d(x^{\prime
},v_{j})$ implies that
\[
d(x,s)=d(x,v_{j})+d(v_{j},s)\not =d(x^{\prime},v_{j})+d(v_{j},s)=d(x^{\prime
},s),
\]
and the claim follows. Suppose therefore that $j\in\{0,i\}$, say $j=0$. Since
$x$ and $x^{\prime}$ are distinguished by $v_{j=0},$ then obviously if $s\in
S\cap V(T_{v_{0}})$ does not distinguish $x$ and $x^{\prime}$ it must hold
$s\not =v_{0}.$ Let $P$ be the shortest path connecting vertices $x$ and $s$
and let $v$ be the vertex on path $P$ which is closest to $v_{0}.$ Then we
have%
\begin{align*}
d(x,s)  &  =d(x,v)+d(v,s),\\
d(x^{\prime},s)  &  =d(x^{\prime},v_{i})+d(v_{i},v_{0})+d(v_{0},v)+d(v,s).
\end{align*}
The fact that $d(x,s)=d(x^{\prime},s)$ implies%
\[
d(x,v)=d(x^{\prime},v_{i})+d(v_{i},v_{0})+d(v_{0},v),
\]
where we can add $d(v_{0},v)$ to both sides of equality and then from
$d(v_{0},v)\geq0$ deduce
\begin{equation}
d(x,v_{0})\geq d(x^{\prime},v_{i})+d(v_{i},v_{0}). \label{l6.in}%
\end{equation}

Note that the assumption that three $S$-active vertices on $C$ forming a
geodesic triple implies that there must exist an $S$-active vertex $v_{l}$
such that $1\leq l\leq\left\lfloor g/2\right\rfloor $, so let us denote by
$s^{\prime}$ a vertex from $S\cap T_{v_{l}}$. There are two possibilities, it
is either $l\leq i$ or $l>i$. In the case when $l\leq i$ we have
\begin{align*}
d(x,s^{\prime})  &  =d(x,v_{0})+d(v_{0},v_{l})+d(v_{l},s^{\prime})\\
&  \ge d(x^{\prime},v_{i})+d(v_{i},v_{0})+d(v_{0},v_{l})+d(v_{l},s^{\prime})\\
&  >d(x^{\prime},v_{i})+d(v_{i},v_{l})+d(v_{l},s^{\prime})\\
&  = d(x^{\prime},s^{\prime}),
\end{align*}
which is a contradiction. In the case when $l>i$ from the facts that
$d(x,v_{0})>d(x^{\prime},v_{i})$ and $0<i<l\le\left\lfloor g/2\right\rfloor $
obviously follows $d(x,s^{\prime})>d(x^{\prime},s^{\prime}),$ which is again a contradiction.

Let us now prove that $S$ is a metric edge generator. Suppose to the contrary,
there are two edges $e=xy$ and $e^{\prime}=x^{\prime}y^{\prime}$ which are not
distinguished by $S$. We assume that end-vertices of $e$ and $e^{\prime}$ are
denoted so that $x$ and $x^{\prime}$ are closer to cycle $C$ then $y$ and
$y^{\prime}$ respectively (if there is the difference between those two
distances). Again, from $a(S)\geq2$ and Lemma \ref{l.gt.e} follows that $e$
and $e^{\prime}$ do not belong to the same connected component of $G-E(C)$.

Suppose first that neither $e$ nor $e^{\prime}$ belong to the cycle $C$. Then
let $G_{y}$, $G_{y^{\prime}}$ and $G_{x}$ be the connected components of
$G-\{e,e^{\prime}\}$ containing vertices $y,$ $y^{\prime}$ and $x$
respectively. If $G_{y}$ or $G_{y^{\prime}}$ contain a vertex from $S,$ then
$e$ and $e^{\prime}$ would be distinguished by $S,$ therefore $S\subseteq
V(G_{x})$. But for every $z\in V(G_{x})$ we have $d(e,z)=d(x,z)$ and
$d(e^{\prime},z)=d(x^{\prime},z),$ so $e$ and $e^{\prime}$ must be
distinguished by $S$ since $x$ and $x^{\prime}$ are distinguished by $S,$ so
we obtained a contradiction.

Suppose now that both $e$ and $e^{\prime}$ belong to the cycle $C$. Lemma
\ref{l.gt.e} implies that $e$ and $e^{\prime}$ are distinguished by $S$-active
vertex from $C,$ say $v_{k}$. But then $d(e,v_{k})\not = d(e^{\prime},v_{k})$
implies that for $s\in S\cap V(T_{v_{k}})$ we have
\[
d(e,s)=d(e,v_{k})+d(v_{k},s)\not =d(e^{\prime},v_{k})+d(v_{k},s)=d(e^{\prime
},s),
\]
which is a contradiction.

Suppose finally that exactly one of $e$ and $e^{\prime}$ is contained on $C,$
say $e^{\prime}$. Without loss of generality we may assume $e\in E(T_{v_{0}})$
and $e^{\prime}=v_{i}v_{i+1}$ where $0\leq i\leq\left\lfloor g/2\right\rfloor
-1$. Again, Lemma \ref{l.gt.e} implies that $e$ and $e^{\prime}$ are
distinguished by an $S$-active vertex from $C,$ say $v_{k}$. If $k\not =0$
then the similar argument as in previous case yields that $e$ and $e^{\prime}$
must be distinguished by $S$. Suppose therefore that $k=0$. Let $s\in S\cap
V(T_{0})$, then $s\not =v_{0}$ otherwise $e$ and $e^{\prime}$ would be
distinguished by $s$. Since $s$ does not distinguish $e$ and $e^{\prime}$, the
fact $s\not =v_{i}$ implies $d(e,v_{0})>d(e^{\prime},v_{0})$. Since there is a
geodesic triple of $S$-active vertices on $C,$ there must exist $S$-active
vertex $v_{l}$ for $1\leq l\leq\left\lfloor g/2 \right\rfloor -1$. The fact
$d(e,v_{0})>d(e^{\prime},v_{0})$ implies $d(e,v_{l})>d(e^{\prime},v_{l})$
which further implies $d(e,s^{\prime})>d(e^{\prime},s^{\prime})$ for
$s^{\prime}\in S\cap V(T_{v_{l}})$. Therefore, $e$ and $e^{\prime}$ are
distinguished by $S$ which is a contradiction.
\end{proof}


\section{Vertex and edge dimensions of unicyclic graphs}

In this section, we show that value of the vertex metric dimension and the
value of edge metric dimension of a unicyclic graph can be one of two
consecutive integers, whose values are determined by a formula of $L(G)$ and
$b(G)$. Recall that $b(G)$ denotes the number of branch-active vertices on the
only cycle $C$ in a unicyclic graph $G.$ With the use of Lemmas \ref{l.br1} to
\ref{Prop_gt} we obrain easily the following Theorem \ref{t.1}. Later with
more involved arguments we extend it to cactus graphs.

\begin{theorem}
\label{t.1} Let $G$ be an unicyclic graph. Then each of $\dim(G)$ and
$\mathrm{edim}(G)$ has value $L(G)+\max\{2-b(G),0\}$ or $L(G)+\max\{2-b(G),0\}
+ 1$.
\end{theorem}

\begin{proof}
Let $S$ be a vertex or edge metric generator of $G$ of the smallest possible
size. Notice that $S$ must contain a set of vertices $S^{\ast}\subseteq
V(G)\setminus V(C)$ that is a branch-resolving set for $G$. As we are assuming
that $S$ is a smallest possible set, we may assume that $|S^{\ast}|=L(G)$ and
in particular $a(S^{\ast})=b(G)$. Lemma~\ref{l.br1} implies that at least
$\max\{2-b(G),0\}$ vertices must be introduced to the branch-resolving set
$S^{\ast}$ in order to become a vertex (resp. an edge) metric generator.
Therefore, $S\setminus S^{\ast}$ contains a set $S^{\prime}$ of $\max
\{2-b(G),0\}$ vertices from $C$. Thus, we obtain that $S^{\ast}\cup S^{\prime
}$ must be of order at least $L(G)+\max\{2-b(G),0\}$. This establishes the
lower bound.

Let us now use the result of Lemma~\ref{Prop_gt} to obtain the upper bound. We
use the same notation as above. Notice that the set of active vertices of
$S^{\ast}$ union $S^{\prime}$, i.e. $A(S^{\ast})\cup S^{\prime}$ is of order
at least 2 but this union may or may not contain a geodesic triple. Notice by
introducing to any set of vertices of $C$ of size $\geq2$ a carefully selected
new vertex $x$, we can always assure that the enlarged set $A(S^{\ast})\cup
S^{\prime}\cup\{x\}$ contains a geodesic triple. Now Lemma~\ref{Prop_gt}
implies that $S^{\ast}\cup S^{\prime}\cup\{x\}$ is a vertex (resp. an edge)
metric generator of $G$. As the latter union is of size $L(G)+\max
\{2-b(G),0\}+1$, we establish the upper bound.
\end{proof}

The above theorem gives us promptly the following result.

\begin{corollary}
\label{c.1} Let $G$ be a unicyclic graph. Then $\left\vert \dim
(G)-\operatorname*{edim}(G)\right\vert \leq1$.
\end{corollary}


\section{Metric dimensions in graphs with edge disjoint cycles}

The results for unicyclic graphs from previous sections can now be extended to
graphs with more cycles than one, as long as those cycles are edge disjoint.
Namely, if cycles in a graph are edge disjoint, then for every cycle in such
graph there is a restriction of the graph to a unicyclic subgraph in which
that cycle is the only cycle. Such restrictions are not necessarily disjoint,
but they cover the whole graph. Then, for a set of vertices in a graph we can
also consider a restriction to a unicyclic subgraphs (with few necessary
accomodations) and then apply to it the results from previous sections, which
then yields the conditions under which such set is a metric generator in a
wider graph. In orther to realize all this, we introduce the following more
formal definitions.

We say that $G$ is a \emph{cactus} graph if all cycles in $G$ are pairwise
edge disjoint. Let $C$ be a cycle in a cactus graph $G$ and let $v$ be a
vertex on it. Note that the connected component $T_{v}$ of $G-E(C)$ in the
case of a cactus graph does not have to be a tree, it can be a cactus graph,
i.e. it may contain cycles. Nevertheless, we extend the definition of
branching and branch-active vertices in a similar way. Similarly, we denote by
$b(C)$ the number of branch-active vertices on $C$ (which now cannot be
denoted by $b(G)$ as we did with unicyclic graphs, since cactus graph may have
more than one cycle).

\bigskip

\begin{lemma}
Let $G$ be a cactus graph with $c$ cycles $C_{1},C_{2},\ldots,C_{c}$. Then,
both $\dim(G)$ and $\mathrm{edim}(G)$ are greater or equal than
\[
L(G)+\sum_{i=1}^{c}\max\{2-b(C_{i}),0\}.
\]

\end{lemma}

\begin{proof}
Let $S\subseteq V(G)$ be a set of vertices from $G$ such that $\left\vert
S\right\vert <L(G)+\sum_{i=1}^{c}\max\{2-b(C_{i}),0\}$. We want to prove that
such $S$ can be neither a vertex nor an edge metric generator. Note that the
one of the following must hold for the set $S$:

\begin{itemize}
\item there is a vertex $v$ in $G$ of degree $\geq3$ such that there are two
threads attached to $v$ which do not contain a vertex from $S;$

\item there is a cycle $C$ in $G$ such that at most one vertex on $C$ is $S$-active.
\end{itemize}

In the first case, let $w_{1}$ and $z_{1}$ be the vertices on those two
threads incident to $v,$ then obviously $w_{1}$ and $z_{1}$ are not
distinguished by $S$, therefore $S$ is not a vertex metric generator. The same
holds for edges $vw_{1}$ and $vz_{1}$ which means $S$ cannot be an edge metric
generator either.

In the second case, if there is no $S$-active vertices on $C$ it means
$S=\phi,$ so $S$ cannot be a metric generator. If there is exactly one
$S$-active vertex $v$ on $C$, let $v_{1}$ and $v_{2}$ be the two neighbors of
$v$ on $C$. Then obviously $v_{1}$ and $v_{2}$ are not distinguished by $S$,
consequently $S$ cannot be a vertex metric generator. The same holds for edges
$v_{1}v$ and $v_{2}v$ and thus $S$ cannot be an edge metric generator either.
\end{proof}

\bigskip We say that a path $P$ in a cactus graph $G$ is a \emph{connector} of
cycles $C_{i}$ and $C_{j}$ of $G$ if the end-vertices $u$ and $v$ of $P$
belong to $V(C_{i})$ and $V(C_{j})$, respectively, and $P$ does not share any
other vertex besides $u$ and $v$ with any cycle in $G$ (for the illustration
see Figure \ref{Figure14}). We also say that $P$ \emph{connects} $C_{i}$ and
$C_{j}$ or that it is \emph{incident} to those two cycles. The \emph{domain}
$G_{i}$ of the cycle $C_{i}$ in a cactus graph $G$ is the graph consisting of
the cycle $C_{i}$, all connector paths incident to $C_{i}$ and all threads and
branches incident to either $C_{i}$ or the corresponding connector paths (see
Figure \ref{Figure13}). Note that $G_{i}$ is a unicyclic graph with $C_{i}$
being its only cycle. Also, note that for two distinct cycles $C_{i}$ and
$C_{j}$ their corresponding domain graphs $G_{i}$ and $G_{j}$ may not be
vertex disjoint because there may exist a connector path between cycles
$C_{i}$ and $C_{j}$ in which case their corresponding domains $G_{i}$ and
$G_{j}$ share that connector paths and all threads and branching hanging on
the vertices of that connector. Finally, $G_{i}$ is obviously an isometric
subgraph of $G$.

We say that a vertex $v_{j}$ from a domain $G_{i}$ is a \emph{boundary} vertex
if $v_{j}$ belongs to $C_{j}$ in $G$ where $j\not =i$. Note that the boundary
vertex $v_{j}$ is actually an end-vertex of the connector path between cycles
$C_{i}$ and $C_{j}$. We say that a vertex $v\in V(G)\setminus V(G_{i})$ is on
the \emph{other side} of the boundary vertex $v_{j}$ of $G_{i}$ if the
shortest path from $v$ to $C_{i}$ contains $v_{j}$ (see again Figure
\ref{Figure13}).

So far we have established domains in a cactus graph which are unicyclic
graphs and which cover the whole $G$. Since we are interested in sets of
vertices $S$ in a cactus graph which are metric generators, we also want to
divide such sets into parts which will be metric generators in those domains,
but we cannot simply take a restriction of the set $S$ onto a domain, since
two vertices or two edges from the same domain can in a cactus graph be
distinguished by a vertex $s\in S$ which is outside that domain, because of
which we must add boundary vertices to the restriction. Therefore, for a set
of vertices $S$ in $G$, let $S_{i}$ denote the set obtained from $S\cap
V(G_{i})$ by adding all boundary vertices from the domain $G_{i}$ to it.
Observe that the following holds: if $S$ is a branch-resolving set in $G$ then
$S_{i}$ is a branch-resolving set in $G_{i}$, also if there is a geodesic
triple of $S$-active vertices on every cycle in $G$ then there is a geodesic
triple of $S_{i}$-active vertices on $C_{i}$ in $G_{i}$.

\begin{figure}[h]
\begin{center}
\includegraphics[scale=1.0]{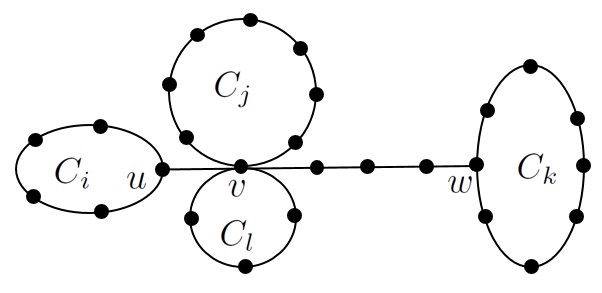}
\end{center}
\par
\caption{An illustration of a connector path: the edge $uv$ is a connector of
cycles $C_{i}$ and $C_{j}$ and also of $C_{i}$ and $C_{l}$. The path from $v$
to $w$ is a connector of cycles $C_{j}$ and $C_{k},$ but the path from $u$ to
$w$ is not a connector of cycles $C_{i}$ and $C_{k}.$ Note that a connector
path can consist of only one vertex ($v$ is a connector of $C_{j}$ and $C_{l}%
$).}%
\label{Figure14}%
\end{figure}

\begin{figure}[h]
\begin{center}
\includegraphics[scale=1.0]{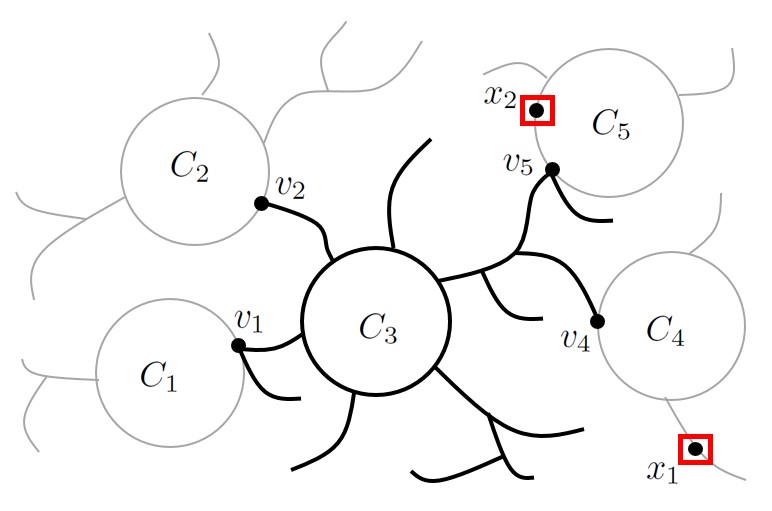}
\end{center}
\par
\caption{An example of a cactus graph with five cycles where the domain
$G_{3}$ of the cycle $C_{3}$ is emphesized. Vertices $v_{1},$ $v_{2},$ $v_{4}$
and $v_{5}$ are the boundary vertices of the domain $G_{3}.$ With respect to
the domain $G_{3},$ the vertex $x_{1}$ is on the other side of the boundary
vertex $v_{4},$ while $x_{2}$ is on the other side of the boundary vertex
$v_{5}.$}%
\label{Figure13}%
\end{figure}

\begin{lemma}
Let $G$ be a cactus graph with $c$ cycles $C_{1},C_{2},\ldots,C_{c}$. Then,
both $\dim(G)$ and $\mathrm{edim}(G)$ are smaller or equal than
\begin{equation}
L(G)+\sum_{i=1}^{c}\max\{2-b(C_{i}),0\}+c. \label{For_upperCactus}%
\end{equation}

\end{lemma}

\begin{proof}
Let us construct a set of vertices $S$ such that $\left\vert S\right\vert $ is
smaller or equal than (\ref{For_upperCactus}), for which we will prove it is
both a vertex and an edge metric generator. The set $S$ will consist of three
parts as the desired bound is the sum of three numbers. Let $S_{a}$ be a
minimum branch-resolving set in $G$, thus $\left\vert S_{a}\right\vert =L(G)$.
Further, let $S_{b}$ consist of $\max\{2-b(C_{i}),0\}$ vertices from every
cycle $C_{i}$ not contained in $S_{a}$. Note that $S_{b}$ contains at most two
vertices from every cycle in $G$ which are chosen so that every cycle in $G$
will have at least two $(S_{a}\cup S_{b})$-active vertices. Finally, note that
a cycle in $G$ may or may not have a geodesic triple of $(S_{a}\cup S_{b}%
)$-active vertices, so we define a third set $S_{c}$ in a following way. Let
$S_{c}$ contain a vertex from every cycle in $G$ which does not have a
geodesic triple of $(S_{a}\cup S_{b})$-active vertices, which vertex is chosen
so that it forms a geodesic triple with the two $(S_{a}\cup S_{b})$-active
vertices which must exist on each cycle. Notice again that for any two
vertices on a cycle we can choose easily the third one so that they form a
geodesic triple. Therefore, $\left\vert S_{c}\right\vert \leq c$. Now we
define $S=S_{a}\cup S_{b}\cup S_{c}$. Obviously
\[
\left\vert S\right\vert \leq L(G)+\sum_{i=1}^{c}\max\{2-b(C_{i}),0\}+c.
\]
Also, note that $S$ is a branch-resolving set with a geodesic triple of
$S$-active vertices on every cycle $C_{i}$ in $G$.

We want to prove that $S$ is both a vertex and an edge metric generator. Let
$x$ and $x^{\prime}$ be any two distinct vertices or two edges in $G.$ In
order to do so, we distinguish the following two cases.

\bigskip\noindent\textbf{Case 1:} \emph{$x,x^{\prime}$ both belong to a same
domain.} Denote this domain by $G_{i}$. Since $S_{i}$ is a branch-resolving
set with a geodesic triple of $S_{i}$-active vertices in a unicyclic graph
$G_{i}$, then Lemma \ref{Prop_gt} implies $S_{i}$ distinguishes $x$ and
$x^{\prime}$ in $G_{i}.$ Recall that $G_{i}$ is a isometric subgraph of $G,$
so if $x$ and $x^{\prime}$ are distinguished by a vertex $s\in S_{i}\cap S$
then $x$ and $x^{\prime}$ are distinguished by $S$ in $G.$ Suppose now that
$x$ and $x^{\prime}$ are distinguished in $G_{i}$ by $s\in S_{i}\backslash S.$
This implies that $s$ is a boundary vertex in $G_{i},$ but this further
implies there must exist a vertex $s_{1}\in S$ on the other side of the
boundary vertex $s$ since otherwise we would have a contradiction with our
assumption that there are at least two $S$-active vertices on every cycle in
$G.$ Now, as $s$ distinguishes $x$ and $x^{\prime}$ then certainly $s_{1}\in
S$ distinguishes them too, so $x$ and $x^{\prime}$ are distinguished by $S$ in
$G.$\begin{figure}[h]
\begin{center}
\includegraphics[scale=1.0]{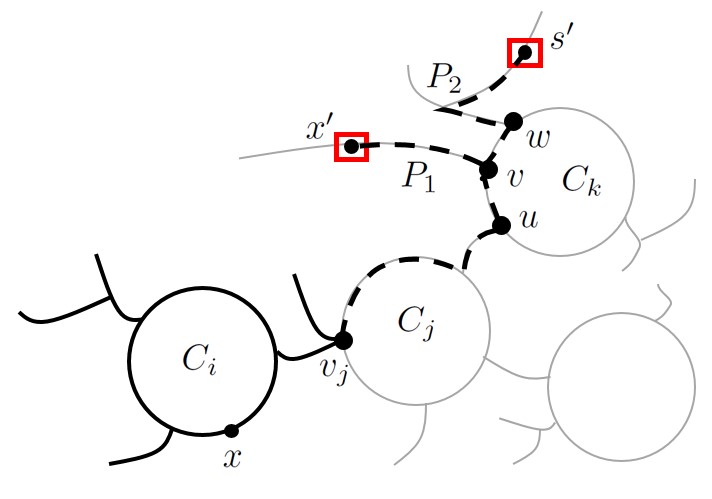}
\end{center}
\par
\caption{An illustration of the graph $G_{i}^{\prime}$: the domain $G_{i}$ of
the cycle $C_{i}$ contains $x$, but does not contain $x^{\prime}$ nor
$s^{\prime},$ so it has to be extended by paths $P_{1}$ and $P_{2}$ which lead
from the boundary vertex $v_{j}$ to $x^{\prime}$ and $s^{\prime}$
respectively.}%
\label{Figure17}%
\end{figure}

\bigskip\noindent\textbf{Case 2:} \emph{$x,x^{\prime}$ do not belong to a same
domain.} Suppose that $x$ belongs to $G_{i}$ and $x^{\prime}$ belongs to
$G_{k}$ where $k\not =i$. So, we assume that $x^{\prime}$ does not belong to
$G_{i}$ and $x$ does not belong to $G_{k}$, otherwise it reduces to the
previous case.

In this case we also want to apply the result for the unicyclic graphs, but as
$x$ and $x^{\prime}$ do not belong to a same domain, now we will have to
expand the domain $G_{i}$ which contains $x$ so that it contains $x^{\prime}$
as well. In order to do so (and the construction that follows is illustrated
in Figure \ref{Figure17}), let $v_{j}\in V(C_{j})$, for $j\not =i,$ be the
boundary vertex of the domain $G_{i}$ closest to $x^{\prime}.$ Let $u$ be the
vertex on cycle $C_{k}$ closest to $v_{j}$ and let $v$ be the vertex on the
cycle $C_{k}$ closest to $x^{\prime}$. Finally, assuming that $l_{k}$ denotes
the length of the cycle $C_{k}$, let $w$ be the $S$-active vertex on $C_{k}$
such that vertices $u,$ $v$ and $w$ belong to a same half of the cycle $C_{k}%
$, where a half of the cycle is any path on the cycle of length $\leq
\lceil(l_{k}-1)/2\rceil$. Note that it is always possible to choose such $w$
because there is a geodesic triple of $S$-active vertices on the cycle
$C_{k}.$ 

Let $P_{1}$ be the shortest path from $x^{\prime}$ to $v_{j}$ in $G,$ note
that $P_{1}$ contains the vertex $v.$ Let $s^{\prime}\in S$ be a vertex not
contained in the domain $G_{i}$ on the same side of the boundary vertex
$v_{j}$ as $x^{\prime},$ which is chosen so that the shortest path $P_{2}$
from $s^{\prime}$ to $v_{j}$ contains the vertex $w.$ This is possible since
$w$ is $S$-active vertex. Finally, let $G_{i}^{\prime}$ be the extension of
the domain $G_{i}$ obtained by adding paths $P_{1}$ and $P_{2}$ to it. Note
that the extension $G_{i}^{\prime}$ is also a unicyclic graph with $C_{i}$
being its only cycle. Also, note that the distances in $G_{i}^{\prime}$ and
$G$ are the same, i.e. $G_{i}^{\prime}$ is a isometric subgraph of $G$, due to
the way in which we chose vertices $u,$ $v$ and $w.$

Let $z$ be the common vertex of paths $P_{1}$ and $P_{2}$ furthest from
$v_{j}$ in $G_{i}^{\prime}$. The construction of $G_{i}^{\prime}$ implies
$z=v$ or $z=w$. Note that the addition of paths $P_{1}$ and $P_{2}$ to $G_{i}$
in order to obtain $G_{i}^{\prime}$ certainly creates two threads hanging at
$z$ in $G_{i}^{\prime}$ which do not contain a vertex from $S_{i}.$ Therefore,
the set $S_{i}$ is not a branch-resolving set in $G_{i}^{\prime}$, so it is
not a good candidate for metric generator. Hence, just as we modified the
domain $G_{i}$ into $G_{i}^{\prime}$, now the set $S_{i}$ also has to be
modified into $S_{i}^{\prime}.$ We define $S_{i}^{\prime}=S_{i}\backslash
\{v_{j}\}\cup\{s^{\prime}\}$, i.e. we replace the boundary vertex $v_{j}$ with
$s^{\prime}$ in the set $S_{i}$ to obtain $S_{i}^{\prime}$. Note that the set
$S_{i}^{\prime}$ defined in this way certainly is a branch-resolving set in
$G_{i}^{\prime}$, because exchanging the vertex $v_{j}$ with $s^{\prime}$ in
$S_{i}$ resolves the problem of two threads hanging at $z$ in $G_{i}^{\prime}%
$. Also, note that there must exist a geodesic triple of $S_{i}^{\prime}%
$-active vertices on $C_{i}$ in $G_{i}^{\prime}$. Thus, according to Lemma
\ref{Prop_gt}, so $S_{i}^{\prime}$ distinguishes $x$ and $x^{\prime}$ in
$G_{i}^{\prime}$. Since $G_{i}^{\prime}$ is isometric subgraph of $G$, this
further implies $S$ distinguishes $x$ and $x^{\prime}$ in $G$ and the case is proven.
\end{proof}

The above two lemmas give us promptly the following result.

\begin{corollary}
\label{c.2} Let $G$ be a cactus graph with $c$ cycles. Then $\left\vert
\dim(G)-\operatorname*{edim}(G)\right\vert \leq c$.
\end{corollary}

The bound $c$ in the above Corollary \ref{c.2} is the best possible as the
graphs for which the both sides of the bound are achieved are presented in
\cite{Knor}.

\section{Further work}

In the paper, we have established that the vertex and edge metric dimensions
of a unicyclic graph differ by at most $1$, more precisely each of them has
its value in
\[
L(G)+\max(2-b(G),0)\quad\hbox{ and/or }\quad L(G)+\max(2-b(G),0)+1.
\]
A possible further work is to research (maybe characterize) when each of these
two values is realized for each of $\dim(G)$ and $\mathrm{edim}(G)$. Our
ongoing investigation shows, beside maybe other things, that it depends also
of the parity of $C$. Let us mention that this will fully resolve when the
difference of $\dim(G)-\mathrm{edim}(G)$ is $-1$, $0$ and $1$.

\begin{problem}
For a unicyclic graph $G$, determine when the difference $\dim(G) -
\mathrm{edim}(G)$ is $-1$, $0$ and $1$.
\end{problem}

One can propose even a more general problem as follows.

\begin{problem}
For a cactus graph $G$ with $c$ cycles, determine when the difference $\dim(G)
- \mathrm{edim}(G)$ is $-c,-c+1,\ldots,-1, 0, 1, 2,\ldots,c-1,c$.
\end{problem}

Another direction is to consider graphs with higher cyclomatic number, which
is defined as $c(G)=|E(G)|-|V(G)|+1$, where $n=|V(G)|$ and $m=|E(G)|$.
Regarding trees have cyclomatic number 0, unicyclic graphs have 1, and other
graphs have $\ge2$. Note that for trees we have $\left\vert \dim
(G)-\operatorname*{edim}(G)\right\vert =0$, unless $G$ is $K_{2}$. Similarly
holds for unicyclic graphs by Corollary~\ref{c.1}. So maybe this relation can
be generalized to graphs with higher cyclomatic number, in particular we
believe the following holds.

\begin{conjecture}
Let $G\ne K_{2}$ be a graph with cyclomatic number $c$. Then
\[
\left\vert \dim(G)-\operatorname*{edim}(G)\right\vert \leq c.
\]

\end{conjecture}

Note that the above conjecture holds for more dense graphs, in particular for
graphs with cyclomatic number $c\ge n-1$ as $n-1\ge\left\vert \dim
(G)-\operatorname*{edim}(G)\right\vert $ for every $G$. Also note that if the
bound of the conjecture holds, then it is tight, see~\cite{Knor} for
infinitely many graphs with $c$ vertex-disjoint cycles each. Moreover there
are presented graphs with $\dim(G)-\mathrm{edim}(G)=c$ as well graphs with
$\mathrm{edim}(G)-\dim(G)=c$.

\bigskip\noindent\textbf{Acknowledgements.}~~The authors acknowledge partial
support Slovenian research agency ARRS program \ P1--0383 and ARRS project
J1-1692 and also Project KK.01.1.1.02.0027, a project co-financed by the
Croatian Government and the European Union through the European Regional
Development Fund - the Competitiveness and Cohesion Operational Programme.

\end{document}